\documentclass[a4paper,10pt]{article}

\usepackage[english]{babel}
\usepackage[latin1]{inputenc}   
\usepackage{amsmath,amsthm}
\usepackage{mathdots}
\usepackage[dvips]{graphicx}
\usepackage{amsfonts,amssymb}
\usepackage{geometry}
\usepackage{hyperref}
\usepackage{stmaryrd}
\usepackage{fancyhdr}
\usepackage{url}
\usepackage{dsfont}

\newtheorem*{conj*}{Conjecture}

\newtheorem*{ack}{Acknowledgements}
\newtheorem*{thm*}{Theorem}

\newtheorem{prop}{Proposition}[section]

\newtheorem{LM}{Lemma}[section]

\newtheorem{thm}{Theorem}[section]

\newtheoremstyle{pourlesremarques}{\topsep}{\topsep}{\normalfont}{}{\bfseries}{.}{ }{}
\theoremstyle{pourlesremarques}
\newtheorem{rem}{Remark}[section]
\newtheorem*{rem*}{Remark}
\newtheoremstyle{pourlesexemples}{\topsep}{\topsep}{\normalfont}{}{\bfseries}{.}{ }{}
\theoremstyle{pourlesexemples}

\newcommand{\w}{\varpi}
\renewcommand{\d}{\delta}
\newcommand{\e}{\epsilon}

\renewcommand{\l}{\lambda}
\newcommand{\C}{\mathbb{C}}

\newcommand{\N}{\mathbb{N}}
\newcommand{\M}{\mathcal{M}}
\newcommand{\Z}{\mathbb{Z}}
\newcommand{\1}{\mathbf{1}}

\newcommand{\D}{\Delta}
\newcommand{\B}{P_\emptyset}

\title {\textbf{Distinction of the Steinberg representation for inner forms of $GL(n)$}}
\author{Nadir MATRINGE\footnote{Nadir Matringe, Universit\'e de Poitiers, Laboratoire de Math\'ematiques et Applications,
T\'el\'eport 2 - BP 30179, Boulevard Marie et Pierre Curie, 86962, Futuroscope Chasseneuil Cedex. Email: Nadir.Matringe@math.univ-poitiers.fr}}

\begin{document}
\maketitle

\begin{abstract}
Let $F$ be a non archimedean local field of characteristic not $2$. Let $D$ be a division algebra of dimension 
$d^2$ over its center $F$, and $E$ a quadratic extension of $F$. If $m$ is a positive integer, to a character $\chi$ of $E^*$, one can attach the Steinberg representation 
$St(\chi)$ of $G=GL(m,D\otimes_F E)$. Let $H$ be the group $GL(m,D)$, we prove that 
$St(\chi)$ is $H$-distinguished if and only if $\chi_{|F^*}$ is the quadratic character $\eta_{E/F}^{md-1}$, where $\eta_{E/F}$ is the character of $F^*$ with kernel the norms of $E^*$. We also get multiplicity one for the space of invariant linear forms.
\end{abstract}

\section{Introduction}

Let $F$ be a non archimedean local field of characteristic not $2$. Let $D$ be a finite dimensional 
division algebra with center $F$, $E$ a quadratic extension of $F$, and $m$ a positive integer. Call 
$\eta_{E/F}$ the character of $F^*$ with kernel the norms of $E^*$. We denote by $H$ 
the group $GL(m,D)$, and by $G$ the group $GL(m,D\otimes_F E)$. We will denote by 
$N_{rd,H}$ the reduced norm of $H$. If the index (the square root of the dimension) $d$ of $D$ over $F$ is even, then there is a division algebra $\D$ with center $E$ of index $d/2$, such that $G=GL(2m,\D)$. If $d$ is odd, then 
$D \otimes_F E$ remains a division algebra $D_E$. Formulated in a different manner, 
our main result says.

\begin{thm*}
Let $St(\1)$ be the Steinberg representation of $G$, and $\chi$ be a character of $H$. Then 
$Hom_H(St(\1),\chi)$ is of dimension zero, except in the case $\chi=\eta_{E/F}^{md-1}\circ N_{rd,H}$, in which case it has dimension $1$.
\end{thm*}

In the context of $F$-quasi split groups, such a result has been conjectured by Prasad in \cite{P01}, and extended in \cite{P16} to all reductive groups, hence the statement above is a particular case of 
the conjectures in \cite{P16}. Notice that for general reductive groups, the definition of the character $\chi$ is involved, see Section 8 of \cite{P16}. In fact \cite{P16} provides 
general conjectures for the shape of Langlands parameters of distinguished representations and the dimension 
of the space of invariant linear forms in terms of those parameters. In particular, Remark 10 of [ibid] explains that $\chi$ does not depend (in a certain sense) on the inner class of the group, so it is not surprising that the statement above looks the same for $GL(md,F)$ and $GL(m,D)$.\\
For $GL(n,F)$, Prasad's conjecture on the Steinberg representation was proved in \cite{AR}, as a consequence of the equality of two forms of the Asai $L$-factor of discrete 
series representations, the proof of which relied on a local-global argument.\\ 
More recently, with a purely local proof, Broussous and Courtès proved Prasad's conjecture for $F$-split 
groups, when $F$ is of residual characteristic different from $2$ (see \cite{B14}, \cite{C15} and \cite{C}). Their method is very general, 
but already becomes very technical when $E/F$ is ramified. To give an idea of it, if $G$ is the $F$-split group in question, 
they use the realisation of $St(\1)$ as the space of harmonic functions on the chambers of the Bruhat-Tits 
building $X_E$ of $G(E)$, and construct an explicit 
linear form using this realisation. The geometry behind the problem of uniqueness of such a form is the understanding 
of the action of $G(F)$ on the chambers of $X_E$, which is hard.\\
In the case of inner forms of $GL(n)$, our method is also local and as naive as possible, we just use the definition 
of the Steinberg representation. The underlying geometry is much simpler, as it is 
amounts to understand the action of $H$ on flag varieties $G/P$, for $P$ a parabolic subgroup 
of $G$. Such an action is well understood thanks to \cite{HW} for example, but in our particular case, 
we use an adhoc method rather than a general one for determining the orbits of $H$ on  $G/P$. 
Maybe in counter-part to the method of Broussous and Courtès, the 
representation theory becomes a bit more involved, though not much as it is essentially Mackey theory and Frobenius reciprocity. 
We think that our method is quite general as well, and this kind of strategy has already been used in different 
contexts (see for instance \cite{FLO}, \cite{M14}, \cite{G}), we will summarise it now.\\
Let $\B$ be the minimal parabolic subgroup of $G$ of upper triangular matrices. Throughout the paper, we will use the notation 
$ind$ for un-normalised parabolic induction. The Steinberg representation $St(\1)$ of $G$ is by definition the quotient of $ind_{\B}^G(\1)$ by the sum of the 
representations $ind_P^G(\1)$, where we can take $P$ amongst the parabolic subgroups containing $\B$ as a proper 
subgroup, and minimal for this property. The existence of a nonzero $(H,\chi)$-equivariant linear form $L$ on $St(\1)$ implies 
that such a form descends from $ind_{\B}^G(\1)$, and Mackey theory together with Frobenius reciprocity shows that 
such a linear form must restrict non trivially to $\mathcal{C}_c^\infty(\B\backslash \B uH)\simeq ind_{\B\cap uHu^{-1}}^{uHu^{-1}}(\1)$, where $\B uH$ is the big (open) cell in 
$G$. Applying Frobenius reciprocity again, one gets that there is at most one (up to scaling) such linear form $L$. It also 
implies that there are only two possible choices $\chi_0$ and $\chi_1$ for $\chi$ (see Propositions \ref{2poss} and 
\ref{onedirection}). For one of these choices, say $\chi_0$, the 
representation $ind_{\B}^G(\1)$ is $(H,\chi_0)$-distinguished thanks to the results of Blanc and Delorme (\cite{BD}), and using Mackey theory and Frobenius reciprocity again, on sees that no representation $ind_P^G(\1)$ can be $(H,\chi_0)$-distinguished, hence 
$St(\1)$ must be $(H,\chi_0)$-distinguished (Propositions \ref{infact1}, \ref{onedirection}, and \ref{versace}). When $d$ is odd, if $m$ is odd, then it follows from Mackey theory and Frobenius reciprocity that if $St(\1)$, hence 
$ind_{\B}^G(\1)$, is $(H,\chi)$-distinguished, then $\chi=\chi_0$ and we are done (Theorem \ref{odd}). In the other cases, we prove that 
when $\chi=\chi_1$, the Steinberg representation $St(\1)$ is not $(H,\chi_1)$-distinguished. If it was, the linear form $L_1$ would descend from $ind_{\B}^G(\1)$, but as $ind_{\B}^G(\1)$ affords 
multiplicity $1$ by Mackey theory and Frobenius reciprocity, the $(H,\chi_1)$-equivariant linear form $L_1$ is given by analytic continuation of an integral according to \cite{BD}. But then, for some well chosen $P$, and well chosen $f\in ind_P^G(\1)$, the computation of $L_1(f)$ reduces to the case where $G$ has split semi-simple rank $1$, and we explicitly show that it does not vanish in this case, hence in the general case (see Theorem \ref{main1}, Propositions \ref{m=2}, \ref{nonzero} and Theorem \ref{main2}). This contradicts the fact that $L_1$ descends to $St(\1)$. In fact, when $d$ is even, we slightly simplify the argument, avoiding 
analytic continuation, though the proof described above would also work in this case.\\
Notice that, in contrast with the method of Broussous and Courtès, our proof is uniform whatever the type of ramification of $E/F$ is. \\
To conclude this introduction, we mention that Dipendra Prasad explained to us that the fact that the ``middle orbits'' cannot support any 
$(H,\chi)$-invariant linear form on $ind_{P_\emptyset}^G(\1)$ is a very general fact, which follows from 
the structure of the groups $P_\emptyset\cap gHg^{-1}$. This is encouraging for a possible generalisation of the method to all reductive groups.

\begin{ack}
I thank the referee for pointing out a mistake in Section \ref{prelim} of the previous version. I thank Ioan Badulescu and Paul Broussous for useful explanations. I thank Dipendra Prasad for his comments on a possible generalisation of this method. I thank R. Kurinczuk and O. Selim for fruitful exchanges. 
I also thank the grant ANR-13-BS01-0012 FERPLAY for financial support.
\end{ack}

\section{General facts and further notations}

We only consider smooth representations on complex vector spaces. 
Let $X$ be a locally compact totally disconnected space, and $L$ a locally compact totally disconnected group acting 
continuously and properly on $X$. If $\chi$ is a character of $L$, we denote 
by $\mathcal{C}_c^\infty(L\backslash X,\chi)$ the space of smooth functions on $X$, with support compact mod $L$, and which transform by $\chi$ under left translation by elements of $L$. If $X$ is a group $Q$ which contains $L$, then we write 
$ind_{L}^Q(\chi)$ for $\mathcal{C}_c^\infty(L\backslash Q,\chi)$, which is a representation of $Q$ by right translation.
 We will use a lot the following two theorems, which are respectively Frobenius reciprocity and Mackey theory for compactly induced representations. The first one is a consequence of Proposition 2.29 of \cite{BZ76}

\begin{prop}\label{Frob}
Let $\chi$ be a character of $L$, then the vector space $Hom_Q(ind_{L}^Q(\chi),\mu)$ is isomorphic to 
$Hom_L(\D\chi,\mu)$, where $\D$ is the quotient of the modulus character of $L$ by that of $Q$.
\end{prop}

The next one is a consequence of Theorem 5.2 of \cite{BZ}. Notice that by Corollary 6.16 of \cite{HW}, if $P$ is a parabolic 
subgroup of $G$, then $P\backslash G/H$ is finite (see also Sections \ref{rep1} and \ref{rep2}). 

\begin{prop}\label{filtr}
Let $P$ be a parabolic subgroup of $G$, and $\mu$ be a character of $P$. Take a set of representatives $(u_1,\dots,u_r)$ 
of $P\backslash G/ H$, ordered such that $X_i=\coprod_{k=1}^i Pu_k H$ is open in $G$ for each $i$. Then $ind_P^G(\mu)$ is filtered by 
the $H$-submodules $\mathcal{C}_c^\infty(P\backslash X_i,\mu)$, and 
$$\mathcal{C}_c^\infty(P\backslash X_i,\mu)/\mathcal{C}_c^\infty(P\backslash X_{i-1},\mu)\simeq 
\mathcal{C}_c^\infty(P\backslash Pu_i H,\mu).$$
\end{prop}

Finally, we recall the following result from \cite{HW}, which is Proposition 3.4 in there. It in particular implies that if 
$P$ contains a minimal $\tau$-split parabolic subgroup (see below), then $\mathcal{C}_c^\infty(P\backslash P H,\mu)$ is 
a subspace of $ind_P^G(\mu)$.

\begin{prop}\label{bigcell}
Let $P$ be a parabolic subgroup of $G$, and $\tau$ be an $F$-rational involution of $G$. The class $PH$ is open if and only 
if $P$ contains a minimal parabolic $\tau$-split subgroup $P'$ (which means that it is sent to its opposite parabolic subgroup $(P')^-$ by $\tau$).
\end{prop}

We already said that $G$ is either $GL(2m,\D)$ when $d$ is even, or $GL(m,D_E)$ when $d$ is odd. We denote by $\B$ 
the minimal parabolic subgroup of $G$ corresponding to upper triangular matrices via this identification. We denote by $M_\emptyset$ 
its Levi subgroup consisting of diagonal matrices. We denote by $N_\emptyset$ the unipotent radical of $\B$. 
We denote by $\Phi$ the roots of the center $Z({M_{\emptyset}})$ of ${M_{\emptyset}}$ acting on the Lie algebra $Lie(G)$, by $\Phi^+$ those corresponding to the restriction 
of this action on $Lie(N_\emptyset)$, and by $\Phi^-$ those corresponding to the restriction 
of this action on $Lie(N_\emptyset^-)$. In particular $Lie(N_\emptyset)=\oplus_{\alpha \in \Phi^+} N_\alpha$ and 
$Lie(N_\emptyset^-)=\oplus_{\alpha \in \Phi^-} N_\alpha$, with 
obvious notations. If $P$ is a parabolic subgroup of $G$ containing $\B$, with standard Levi factor $M$, we denote by $\Phi_M$ 
the roots of $Z({M_{\emptyset}})$ on $Lie(M)$. We define $\Phi_M^+$ and $\Phi_M^-$ in a similar fashion as above.\\

\noindent We denote by $|.|_F$ the normalised absolute value of $F$, and by $|.|_E$ that of $E$.

\section{The case $d$ even}

\subsection{Preliminary remarks on $D\otimes_F E$}\label{prelim}

Let $D$ be a central division algebra of dimension $d^2$ over $F$ with $d=2\d$ even. In this case $E$ identifies with a sub-field of 
$D$, and $D\otimes_F E\simeq \M(2,\D)$, where $\D$ is the centraliser $Z_{D}(E)$ of $E$, and is
 central division algebra of dimension $\d^2$ over $E$. In particular, there is an involution $\sigma$ of 
 $\M(2,\D)$ corresponding to the involution $d\otimes z\mapsto d\otimes \overline{z}$ of $D\otimes_F E$ which fixes a subalgebra of $\M(2,\D)$ isomorphic to $D$. In order to compute explicitly the double cosets in the next section, it will be convenient to have a less abstract description of this involution of $\M(2,\D)$, which in particular depends on the choice of the isomorphism between $D\otimes_F E$ and $\M(2,\D)$. This is what this section will be devoted to. 
 
 First, we recall that by the Skolem-Nother theorem, the Galois involution of $E$ over $F$ is induced by an inner automorphism $int_\e:d\mapsto \e d \e^{-1}$ of $D$, for $\e\in D^\times$. Notice that in particular, $\e$ does not belong to $\D$, hence that $D=\D\oplus \e. \D$, i.e. that $(1,\e)$ is a basis of $D$ as a right $\D$-vector space. We also recall why $D\otimes_F E$ and $\M(2,\D)$ are isomorphic. Indeed, for $d_0\in d$, denoting by $\l(d_0):d\mapsto d_0 d$ and $\rho(d_0):d\mapsto d d_0$ the left and right translations by $d_0$ on $D$, there is a canonical isomorphism $\phi_1$ between $D\otimes_F E$ and $End(D)_{\D}$, defined by $$\phi_0:d\otimes e\mapsto \l(d)\circ \rho(e).$$
 
 Then we set $$\phi_1: u\mapsto Mat_{(1,\e)}(u)$$ to identify $End(D)_{\D}$ with $\M(2,\D)$, and set 
 $\phi=\phi_1\circ \phi_0$. Via these identifications, an automorphism $\psi$ of $\M(2,\D)$ will correspond to an automorphism $\psi_0$ of $D\otimes_F E$, and $\psi_1$ of $End(D)_{\D}$. In particular, the Galois involution $\sigma_0$ of $D\otimes_F E$, corresponds to the involution $\sigma_1$ of $End(D)_\D$ and $\sigma$ of $\M(2,\D)$. We denote by $\theta$ the $F$-linear (and in fact $E$-semi linear) automorphism of $\M(2,\D)$ which is given by applying $int_\e$ to every coefficient of a matrix in $\M(2,\D)$. We also set 
 $$s_\e=\phi_1(\l(\e))=\begin{pmatrix} & \e^2 \\ 1 & \end{pmatrix}\in \M(2,\D).$$

\begin{LM} We have the equality of involutions $\sigma=\theta^{-1}\circ int_{s_\e}=int_{s_\e} \circ \theta^{-1}.$
\end{LM}
\begin{proof}
Write $\sigma'$ for $\theta^{-1} \circ int_{s_\e}$. It will be more convenient to work with 
$End(D)_\D$. First we notice that $\sigma_1$ fixes $\l(D)$ by definition, and we claim that $\sigma'_1  =
\theta_1^{-1} \circ int_{\l(\e)}$ as well. Indeed, a map $u$ in $End(D)_\D$ belongs to $\l(D)=End(D)_D$ if and only if it commutes with $\rho(\e)$, i.e. if and only if 
$\rho(\e)^{-1}\circ u= u \circ \rho(\e)^{-1}$. This can be rewritten as 
$\l(\e)^{-1}\circ int_\e \circ u= u \circ \l(\e)^{-1}\circ int_\e\in End_F(D)$, which can in turn be written as 
$\sigma'_1(u)=u$ (as $\theta_1=int_{int_\e}$).
Now notice that both $\sigma_1$ and $\sigma'_1$ are $E$-semi linear, hence $\sigma_1\circ (\sigma'_1)^{-1}$ is $E$-linear, hence by the Skolem-Noether theorem, it is an inner automorphism of 
$\mathcal{M}(2,\D)$. It is thus equal to $int_l$, for some $l$ in $(End(D)_\D)^\times$. But as $\sigma_1\circ (\sigma'_1)^{-1}$ fixes $\l(D)$, one has $l\circ \lambda(d)=\lambda (d)\circ l$ for all $d\in D$, i.e. 
$$l\in End_D(D)_\D=End_D(D)\cap End(D)_\D=\rho(D)\cap End(D)_\D= \rho(Z_D(\D))=\rho(E).$$ In particular 
$int_l$ is the identity of $End(D)_\D$, hence $\sigma_1=\sigma_1'$, and this proves the first equality of the lemma. The second comes from the fact that $s_\e$ commutes with $\e.I_2$ in $\mathcal{M}(2,D)$.
\end{proof}

We denote by $N_{rd,E}$ the reduced norm on $GL(k,\D)$, and by $N_{rd,F}$ that on $GL(k,D)$. We denote by 
$\nu_E$ the positive character $\nu_E:g\mapsto |N_{rd,E}(g)|_E$ of $GL(k,\D)$, and by 
$\nu_F$ the positive character $\nu_F:g\mapsto |N_{rd,F}(g)|_F$ of $GL(k,D)$.

\subsection{Representatives of $P\backslash G/H$}\label{rep1}

We set $n=2m$, and $\theta$ the $E$-semi linear automorphism of 
$\mathcal{M}(n,\D)$ which acts by $int_\e$ on each coefficient of a matrix in 
$\mathcal{M}(n,\D)$. We let $V$ be the right vector space $D^m$ with canonical basis $\mathcal{B}_c=(e_1,\dots,e_m)$, which we identify 
to $\D^n$ via the basis $\mathcal{U}=(u_1,\dots,u_n)$, where $u_i=e_i$ if $i=1,\dots,m$, and $u_{m+i}=e_{m+1-i}.\epsilon$. In this basis, according to Section \ref{prelim}, the space 
$End(V)_\D$ identifies with $\M_n(\D)$, and $End(V)_D$ with the fixed points in $\M_n(\D)$ of the involution 
$$\sigma=\theta^{-1} \circ int_{s_\e,n}=int_{s_\e,n}\circ \theta^{-1},$$ where 
$$s_{\e,n}=Mat_{\mathcal{U}}(\l(\e))=
\begin{pmatrix} & & & & & \e^2 \\ & & & & \iddots & \\ & & & \e^2  & &  \\ & & 1 & & &  \\ & \iddots & & & &  \\ 1 & & & & & \end{pmatrix}.$$ 
When the size of this matrix is clear, we will simply write $s_\e$ for $s_{\e,n}$. In particular the group $H=G^\sigma$ identifies with $GL(m,D)$. 
Clearly, a right $\D$-subspace $W$ of $V$ is a right $D$-subspace if and only if $\rho(\e)(W)=W.\e \subset W$.\\
If $\overline{n}=(n_1,\dots,n_r)$ is a partition of $n=2m$ (i.e. $n=n_1+\dots+n_r$ with $n_i>0$), we denote by $V_i^0$ the subspace 
$Vect(e_1,\dots, e_{n_1+\dots+n_i})_\D$ of $V$. 
We recall that if $P=P_{(n_1,\dots,n_r)}$ is a standard parabolic subgroup of $G$, then $G/P$ identifies to the flags 
$0\subset V_1\subset \dots \subset V_r=V$ such that $dim(V_i)=n_1+\dots+n_i$ via the map 
$gP\mapsto (gV_1^0,\dots, gV_r^0)$. If $\mathcal{F}=0\subset V_1\subset \dots \subset V_r=V$ is a flag, for $i\leq j$, we fix a complement 
$S_{i,j}$ of the space $V_i\cap V_{j-1}.\e+V_{i-1}\cap V_j.\e$ in $V_i\cap V_j.\e$. Moreover, if $i=j$, we choose $S_{i,i}$ 
to be $\rho(\e)$-stable (or equivalently a right $D$-subspace). We then set $S_{j,i}=S_{i,j}.\e$. One checks that $V$ 
decomposes as 
$$V=\oplus_{i,j} S_{i,j}$$ and that $$V_i=\oplus_{k=1}^i\oplus_{l=1}^r S_{k,l}.$$
Notice that as the spaces $S_{i,i}$ are right $D$-vector spaces, their dimension $n_{i,i}=2m_{i,i}$ over $\D$ is even. 
Notice as well that the dimensions $n_{i,j}$ (over $\D$) of the spaces $S_{i,j}$ are uniquely determined by $\mathcal{F}$. Suppose 
that two flags $\mathcal{F}=0\subset V_1\subset \dots \subset V_r=V$ and 
$\mathcal{F}'=0\subset V_1'\subset \dots \subset V_r'=V$ are in the same $H$-orbit (i.e. 
$\mathcal{F}'=h\mathcal{F}$), then clearly one can choose $S'_{i,j}=hS_{i,j}$ for all $i$, hence for all $(i,j)$, 
we have $n'_{i,j}=n_{i,j}$. Conversely, suppose that $\mathcal{F}'$ and $\mathcal{F}'$ are two flags such that for all $(i,j)$, 
one has $n'_{i,j}=n_{i,j}$. For $i=j$, 
as $S_{i,i}$ and $S'_{i,i}$ are both right $D$-vector spaces, then one can choose an isomorphism 
$u_{i,i}\in Iso(S_{i,i},S'_{i,i})_D$. If $i<j$, then take an element $u_{i,j}\in Iso(S_{i,j},S'_{i,j})_\D$, and define 
$u_{j,i}= \rho(\e) \circ u_{i,j} \circ \rho(\e)^{-1}\in  Iso(S_{j,i},S'_{j,i})_\D$, then we have 
$u_{i,j}\oplus u_{i,j}\in Iso(S_{i,j}\oplus S_{j,i},S'_{i,j}\oplus S'_{j,i})_D$ because it clearly commutes with $ \rho(\e)$.\\
Hence if $\overline{n}=(n_1,\dots,n_r)$ is a partition of $n$, we define $I(\overline{n})$ to be the set of symmetric matrices 
$s=(n_{i,j}) \in \mathcal{M}_r(\N)$ with positive integral entries, even on the diagonal, and the sum of the $i$-th row being equal to $n_i$. Let 
$s=(n_{i,j})$ be such a matrix, so that one can write $n$ as the ordered sum
$$n=m_{1,1}+n_{1,2}+\dots+n_{1,r}+ m_{2,2}+ n_{2,3}+ \dots+n_{2,r}+ m_{3,3}+ \dots +m_{r-1,r-1}+ n_{r-1,r}+m_{r,r}$$
$$+m_{r,r}+n_{r,r-1}+m_{r-1,r-1}+\dots+m_{3,3}+n_{r,2}+\dots +n_{3,2}+ m_{2,2}+ n_{r,1}+ \dots+ n_{2,1}+m_{1,1},$$ 
but one can also write it using the lexicographical ordering: 
$$n=n_{1,1}+\dots+n_{1,r}+n_{2,1}+\dots+n_{2,r}+\dots+n_{r,1}+\dots+n_{r,r}.$$
We denote by $w_s$ the matrix of the permutation (still denoted $w_s$) defined as follows:\\
If $i\in \llbracket{1,\dots,r \rrbracket}$, then for $k\in \llbracket{1,\dots,m_{i,i}\rrbracket}$, we set
$$w_s(m_{1,1}+\dots+m_{i-1,i-1}+n_{i-1,i}+\dots+n_{i-1,r}+k)=n_{1,1}+\dots+n_{i-1,1}+\dots+n_{i-1,r}+k,$$ and 
$$w_s(m_{1,1}+\dots+m_{i+1,i+1}+n_{r,i}+\dots+n_{i+1,i}+k)=n_{1,1}+\dots+n_{i-1,1}+\dots+n_{i-1,r}+k+m_{i,i}.$$
If $i<j$, for $k\in \llbracket{1,\dots,n_{i,j}\rrbracket}$ we set 
$$w_s(m_{1,1}+ \dots+ m_{i,i}+ n_{i,i+1}+ \dots +n_{i,j-1}+k)=n_{1,1}+ \dots + n_{1,r}+\dots+ n_{i,1}+ \dots+ \dots +n_{i,j-1}+k $$ and 
$$w_s(m_{1,1}+ \dots+ m_{i+1,i+1}+ n_{r,i} +\dots+ n_{j+1,i}+k)= n_{1,1}+ \dots + n_{1,r}+\dots+ n_{j,1}+ \dots+ \dots +n_{j,i-1}+k .$$ 
It follows from the definition of $w_s$, that the flag $\mathcal{F}=0\subset V_1\subset \dots \subset V_r=V$, with 
$V_i=w_s^{-1}(V_i^0)$, and the fact that $\rho(\e)(Vect(u_i))=Vect(u_{n+1-i})$, that $dim(S_{i,j})=n_{i,j}$ for all couples $(i,j)$. We thus proved the following result:

\begin{prop}\label{representatives}
Let $\overline{n}$ be a partition of $n$, and $P=P_{\overline{n}}$, then $G=\coprod_{s\in I(\overline{n})} Hw_s^{-1}P$, hence 
$G=\coprod_{s\in I(\overline{n})} Pw_s H$. In particular, it follows from the proof of the proposition, or from its statement, that the set of representatives of $P\backslash G/H$ does not depend on the division algebra $D$ of even index, i.e. the identity map of the set $\{w_s,\ s\in I(\overline{n})\}$, induces a bijection from $P\backslash G/H$ to $P_{\overline{n}}(E)\backslash GL(n,E)/GL(m,\mathbb{H})$, where $\mathbb{H}$ is the quaternion algebra over $F$.
\end{prop}

\begin{rem}\label{indep1}
We will see in Section \ref{odd}, Remark \ref{indep2}, that the set of representatives of $P\backslash G /H$ above "naturally" embeds as a subset of $P_{\overline{n}}(E)\backslash GL(n,E)/GL(n,F)$.
\end{rem}

We recall that $\sigma$ is the involution $int_{s_\e} \circ \theta^{-1}$, so that $H=G^\sigma$. Notice that 
$\theta^{-1}$ fixes $\mathcal{M}(n,F)$, hence $w_s$. For 
$s\in I(\overline{n})$, we denote by $t_s=w_s\sigma(w_s^{-1})s_\e=w_s s_\e w_s^{-1}$, which is a monomial matrix 
(in $N_{G}({M_{\emptyset}})$). We denote 
by $\tau_s$ the element of order $2$, which is the image of $t_s$ in $\mathfrak{S}_n=N_{G}({M_{\emptyset}})/{M_{\emptyset}}$, in fact as a permutation matrix, one has $\tau_s=w_sww_s^{-1}$. With these choices, 
the group $w_s H w_s^{-1}$ is the fixed points of the involution 
$$\sigma_s:x\mapsto t_s \theta^{-1}(x)t_s^{-1}=\theta^{-1}(t_s x t_s^{-1}).$$
We want explicit information about $\tau_s$. We write $I=\llbracket 1,\dots, n \rrbracket$ as the ordered (in the 
sense that the elements in one set are smaller than those in the set written after)
disjoint union 
$$I=I_{1,1}\cup I_{1,2}\cup \dots \cup I_{1,r} \cup \dots \cup I_{r,1}\cup \dots \cup I_{r,r-1}\cup I_{r,r},$$ 
with $I_{i,j}$ of length $n_{i,j}$. Then one checks that $\tau_s$ is the involution of $I$, which stabilises each $I_{i,i}$, acting on it as the symmetry with respect to its midpoint, and which stabilises $I_{i,j}\cup I_{j,i}$ (for $i<j$) and acts 
on this union of intervals as the symmetry with center the midpoint of the interval joining the left end of $I_{i,j}$ and the right 
end of $I_{j,i}$. Any $s\in I(\overline{n})$ can be interpreted as a sub-partition (taking only the nonzero 
$n_{i,j}$'s) of $(n_1,\dots,n_t)$, and we write 
$P_s=M_sN_s$ the corresponding standard parabolic subgroup (and its standard decomposition) of $G$ contained in $P$. We follow 
\cite{M11} and \cite{JLR} to study the group $P\cap w_s H w_s^{-1}$.  First, 
as $\tau_s$ exchanges the intervals $I_{i,j}$ and $I_{j,i}$ for $j\neq i$, and stabilises the intervals $I_{i,i}$, a straightforward analogue of Remark 3.1 of \cite{M11} gives the following lemma (where $P=MN$ is the standard decomposition of $P$). 

\begin{LM}\label{boozer}
For $s\in I(\overline{n})$, one has $\tau_s(\Phi_M^{-})\subset \Phi^-$, $\tau_s(\Phi_M^{+})\subset \Phi^+$.
\end{LM}

We now obtain the following decomposition of $P\cap w_s H w_s^{-1}$.

\begin{prop}
For any $s\in I(\overline{n})$, one has $P\cap w_s H w_s^{-1}=P_s\cap w_s H w_s^{-1}$, and 
$P_s\cap w_s H w_s^{-1}$ is the semi direct product of $M_s\cap w_s H w_s^{-1}$ and $N_s\cap w_s H w_s^{-1}$.
\end{prop}
\begin{proof}
It is enough to see that $P\cap t_sP t_s^{-1}\subset P_s$, as $P\cap w_s H w_s^{-1}\subset P\cap t_s P t_s^{-1}$ ($w_s H w_s^{-1}$ 
being the fixed points of the involution $\sigma_s: g\mapsto t_s\theta^{-1}(g) t_s^{-1}$). In fact it is enough to show 
that the intersection $Lie(P)\cap t_s Lie(P) t_s^{-1}\subset Lie(P_s)$ (take the invertible elements to go back to the groups). 
Decomposing $Lie(P)$ as $(Lie(N_s^-)\cap Lie(M))\oplus Lie (P_s)$, as $t_s$ is a Weyl element, one has 
$$Lie(P)\cap t_s Lie(P) t_s^{-1}=(Lie(N_s^-)\cap Lie(M)\cap t_s Lie(P) t_s^{-1})\oplus (Lie (P_s)\cap t_s Lie(P) t_s^{-1}).$$ Hence 
it is sufficient to see that $Lie(N_s^-)\cap Lie(M)\cap t_s Lie(P) t_s^{-1}=\{0\},$ or what is equivalent, that 
$t_s^{-1}(Lie(N_s^-)\cap Lie(M))t_s\cap Lie(P)=\{0\}$, and it is in fact enough to replace $Lie(N_s^-)\cap Lie(M)$ by 
$Lie(N_\alpha)$ for $\alpha$ in $\Phi^- -\Phi_{M_s}^-$ in the previous equality, again because $t_s$ is a Weyl element. Now 
$\tau_s(\alpha)$ is not in $\Phi_M$, because otherwise it would be in $\Phi_{M_s}=\Phi_M\cap \tau_s(\Phi_M)$, and thanks to 
Lemma \ref{boozer}, $\tau_s(\alpha)$ is in $\Phi^-$, so it is in $\Phi^- -\Phi_M$, hence 
$t_s Lie(N_\alpha) t_s^{-1} \cap  Lie(P)=\{0\}$. Now if 
$p\in P_s\cap  w_s H w_s^{-1}$, and write it $p_s=m_s n_s$ with $m_s\in M$ and $n_s\in N$, then 
$\theta^{-1} (m_s)\theta^{-1} (n_s)=\theta^{-1} (p_s)=t_s^{-1} p_s t_s=(t_s^{-1} m_s t_s)(t_s^{-1} n_s t_s)$. But 
$t_s$ normalises $M_s$, in particular $P\cap t_s^{-1} N_s t_s \subset N_s$, thus
 $\theta^{-1}(m_s)=t_s^{-1} m_s t_s$ and $\theta^{-1}(n_s)=t_s^{-1} n_s t_s$, and this ends the proof of the proposition.
\end{proof}    
       
The group $M_s^{\sigma_s}=M_s\cap w_s Hw_s^{-1}$ is explicitly described as follows: 
an element $m\in M_s^{\sigma_s}$ is of the form 
$$a=diag(a_{1,1},a_{1,2},\dots,a_{r,r-1},a_{r,r}),$$ 
with $a_{i,i}\in G_{n_{i,i}}(\D)$ satisfying $\theta(a_{i,i})=s_\e a_{i,i}s_{\e}^{-1}$ (here $s_\e=s_{\e,n_{i,i}}$), $a_{i,j}\in G_{n_{i,j}}(\D)$ satisfying 
$\theta(a_{i,j})=\e^2. w  a_{i,j} w^{-1} .\e^{-2}$ if $i<j$, where $w$ is the anti-diagonal matrix with ones on the second diagonal, hence $\theta(a_{i,j})=w a_{i,j} w^{-1}$ if $i>j$ (remember that $\theta$ is not an involution, but that $\theta^2=int_{\e^2.I_n}$).\\

Exactly as in the proof of Proposition 4.4 of \cite{M11} (which is itself up to notational modifications Proposition 2.2 of \cite{JLR}), 
one has the following equality.

\begin{prop}\label{modulus}
$(\d_{P_s^{\sigma_s}})_{|M_s^{\sigma_s}}=(\d_{P_s}^{1/2})_{|M_s^{\sigma_s}}.$
\end{prop}

\subsection{Distinguished Steinberg representations}\label{dist1}

Let $\chi$ be a character $E^*$, we denote by $\tilde{\chi}=\chi\circ N_{rd,E}$ the associated character of $\D^*$. We denote by 
$\mu_{\chi}$ the character $\chi\circ N_{rd,G}$ (and its restriction to any subgroup of $G$), in particular 
${\mu_{\chi}}_{|M_\emptyset}$ is the character $\tilde{\chi}\otimes \dots \otimes \tilde{\chi}$. We will use several times the following fact, which is a consequence of Propositions \ref{Frob} and \ref{filtr}. 

\begin{prop}\label{mackey}
Let $P$ be a standard parabolic subgroup of $G$ corresponding to a partition $n$. 
Suppose that $ind_P^G(\mu_\chi)$ is $H$-distinguished, 
then there is $s\in I(\overline{n})$ such that 
$${\mu_\chi}_{|M_s^{\sigma_s}}=(\d_{P_s^{\sigma_s}})_{|M_s^{\sigma_s}}=(\d_{P_s}^{1/2})_{|M_s^{\sigma_s}},$$ 
\end{prop}
\begin{proof}
According to Proposition \ref{filtr}, the induced representation $ind_P^G(\mu_\chi)$ possesses a filtration, each sub-quotient of 
which is of the form $\mathcal{C}_c^\infty(P\backslash Pw_s H)\simeq ind_{P^{\sigma_s}}^{G^{\sigma_s}}(\mu_\chi)$ for some 
$s\in I(\overline{n})$. A nonzero $H$-invariant linear form on $ind_P^G(\mu_\chi)$ must thus induce 
a nonzero $H$-invariant linear form on one of these sub-quotients. As $P^{\sigma_s}$ is equal to $P_s^{\sigma_s}$, and $G^{\sigma_s}$ is unimodular, the statement is a consequence 
of Proposition \ref{Frob}.  
\end{proof}

 We denote by $St(\chi)$ the Steinberg representation 
$ind_{P_\emptyset}^G(\mu_\chi)/S$, where $S=\sum_P ind_P^G(\mu_\chi)$, where the parabolic subgroups in the sum correspond to a partition 
$\overline{n}$ of $n$, with all $n_i$'s equal to $1$, except one of them which is $2$. First, we observe that if $St(\chi)$ is $H$-distinguished, then one has $\overline{\chi}=\chi^{-1}$.

\begin{prop}\label{2poss}
If $St(\chi)$ is $H$-distinguished, then $ind_{\B}^G(\mu_\chi)$ is distinguished, and this implies that 
$\overline{\chi}=\chi^{-1}$, i.e. $\chi_{|F^*}=\1$ or $\eta_{E/F}$. Moreover, only the open orbit 
$P_\emptyset H$ supports an $H$-invariant linear form, and $Hom_H(St(\chi),\1)$ is of dimension $\leq 1$.
\end{prop}
\begin{proof}
A nonzero $H$-invariant linear form $L$ on $St(\chi)$ lifts to $ind_{\B}^G(\mu_\chi)$, hence the first part of the statement. 
Now suppose that $ind_{P_\emptyset}^G(\mu_\chi)$ is distinguished. Then there is 
$s\in I(1,\dots,1)$, such that 
${\mu_\chi}_{|M_s^{\sigma_s}}=(\d_{P_{\emptyset,s}}^{1/2})_{|M_s^{\sigma_s}}$. Necessarily, one has
 $P_{\emptyset,s}=P_\emptyset$ because $P_\emptyset$ is a minimal parabolic subgroup. According to the description of the group $M_s^{\sigma_s}$ before 
Proposition \ref{modulus}, there is an involution $\tau_s$ of $\{1,\dots,n\}$ with no fixed points, such that 
$M_s^{\sigma_s}$ is the group of matrices $diag(a_{1,\tau_s(1)},\dots,a_{n,\tau_s(n)})$, such that 
$a_{\tau_s(i),i}= \theta(a_{i,\tau_s(i)})$ if $\tau_s(i)>i$. Hence 
the equality ${\mu_\chi}_{|M_s^{\sigma_s}}=(\d_{P_\emptyset}^{1/2})_{|M_s^{\sigma_s}}$, which reads 
$$\prod_{i<\tau_s(i)}\tilde{\chi}(a_{i,\tau_s(i)})\tilde{\chi}(\theta(a_{i,\tau_s(i)}))= \prod_{i=1}^n
\nu_E(a_{i,\tau_s(i)})^{\d(n-2i+1)},$$ is possible if and only if $\tau_s(i)=n+1-i$ for all $i$, i.e. if $s$ is anti-diagonal, which amounts to say that $w_s=1$. In this case 
$${\mu_\chi}(diag(a_1,\dots,a_m,\theta(a_m),\dots,\theta(a_1)))= \prod_{i=1}^m\tilde{\chi}(a_i\theta(a_i))=\prod_{i=1}^m\chi(N_{rd,E}(a_i))\chi(\overline{N_{rd,E}(a_i)})\,$$
and $$\d_{P_\emptyset}^{1/2}(diag(a_1,\dots,a_m,\theta(a_m),\dots,\theta(a_1)))=1,$$ which 
implies that $\overline{\chi}=\chi^{-1}$. Moreover, it follows from the proof 
of Proposition \ref{mackey} that if $L$ is a nonzero $H$-invariant linear form on $ind_{P_\emptyset}^G(\mu_\chi)$, 
then it restricts as a nonzero $H$-invariant linear form on 
$\mathcal{C}_c^\infty({P_\emptyset}\backslash P_\emptyset H,\mu_\chi)\simeq ind_{P_\emptyset\cap H}^H(\mu_\chi)$, which supports up to scaling at most 
one such linear form by Frobenius reciprocity law. The multiplicity at most one statement follows.

\end{proof}

We will show that $St(\chi)$ is distinguished if and only if $\chi_{|F^*}=\eta_{E/F}$. First we show one implication. 

\begin{prop}\label{infact1}
Suppose that $\chi_{|F^*}=\eta_{E/F}$, then $St(\chi)$ is distinguished.
\end{prop}
\begin{proof}
We claim that $ind_{P_\emptyset}^G(\mu_\chi)$ is distinguished. As $\sigma(P_\emptyset)=P_\emptyset^-$, and as $\d_{P_\emptyset}^{-1/2}\mu_\chi$ is a ${M_{\emptyset}}^\sigma$-distinguished character of ${M_{\emptyset}}$, Theorem 2.8. of \cite{BD} applies, and it implies that $ind_{P_\emptyset}^G(\mu_\chi)$ is distinguished. Call $L$ 
such a nonzero $H$-invariant linear form on $ind_{P_\emptyset}^G(\mu_\chi)$, it is enough to show that 
$L$ vanishes on $ind_P^G(\mu_\chi)$, for all standard parabolic subgroups $P$ of type $\overline{n}$, with all 
$n_i$'s equal to $1$, except one being $2$. If 
$ind_P^G(\mu_\chi)$ was $H$-distinguished, then by Proposition \ref{mackey}, there would be 
$s\in I(\overline{n})$ such that ${\mu_\chi}_{|M_s^{\sigma_s}}=(\d_{P_s}^{1/2})_{|M_s^{\sigma_s}}.$ 
The group $P_s$ is equal to $P_\emptyset$ or $P$. It is equal to $P$ if and only if, if $n_i$ is the term equal to $2$ in $\overline{n}$, then $n_{i,i}=2$. In this case, the equality $\mu_\chi=\d_{P}^{1/2}$ on matrices 
$diag(1,\dots,1,g_{i,i},1,\dots,1)$, with $g_{i,i}\in D^*\subset GL(2,\D)$, is impossible since the character 
on the left side takes negative values, whereas that on the right side does not. Hence we are 
left with the case $P_s=P_\emptyset$, so that $\mu_\chi=\d_{P_\emptyset}^{1/2}$ must agree on $M_s^{\sigma_s}$. Let $i$ be the integer such that $n_i=2$. Then $n_i=n_{i,k}+n_{i,l}$, with $k<l$ both different from $i$. The only way that $\mu_\chi$ which is of the form 
$\tilde{\chi}\otimes \dots \otimes \tilde{\chi}$, and $\d_{P_\emptyset}^{1/2}(diag(a_1,\dots,a_n))=\prod_{i=1}^n \nu_E(a_i)^{\d(2i-1-n)/2}$ 
can agree on $M_s^{\sigma_s}$, would be that $k=n+1-i$ and $l=n+1-i$ as well. This is not possible.
\end{proof}

Now we prove the main result of this section.

\begin{thm} \label{main1}
The representation $St(\chi)$ is distinguished if and only if $\chi_{|F^*}=\eta_{E/F}$, in which case 
$Hom_H(St(\chi),\eta_{E/F})$ is of dimension $1$.
\end{thm}
\begin{proof}
According to Proposition \ref{2poss}, it remains to show that if $\chi_{|F^*}=1$, then the 
representation $St(\chi)$ is not distinguished. We suppose that $\chi_{|F^*}=1$ and that 
$St(\chi)$ is distinguished. The second condition implies that $Ind_{P_\emptyset}^G(\mu_\chi)$ is distinguished, and the proofs 
of Proposition \ref{mackey} and \ref{2poss} show that any nonzero $H$-invariant linear form $L$ on $Ind_{P_\emptyset}^G(\mu_\chi)$ 
vanishes on all $H$-sub-quotients $\mathcal{C}_c^\infty(P_\emptyset\backslash P_\emptyset w_sH,\mu_\chi)$, except when $w_s=1$. In particular, 
as 
$$\mathcal{C}_c^\infty(P_\emptyset\backslash P_\emptyset H,\mu_\chi)= 
\mathcal{C}_c^\infty(P_\emptyset\backslash P_\emptyset H,\mu_\chi) \simeq \mathcal{C}_c^\infty(P_\emptyset\cap H\backslash H),$$ by Frobenius reciprocity, 
the space $Hom_H(\mathcal{C}_c^\infty(P_\emptyset\backslash P_\emptyset H),\1)$ is one dimensional, and up to a nonzero scalar, 
$L_{|\mathcal{C}_c^\infty(P_\emptyset\backslash P_\emptyset H)}$ is given by 
$$L:\phi\mapsto \int_{P_\emptyset\cap H\backslash H} \phi(h)dh$$ (notice 
that the group $P_\emptyset\cap H$ is equal to ${M_{\emptyset}}$, hence unimodular just as $H$, and the integral is thus well defined). As 
$\sigma(P_\emptyset)=P_\emptyset^-$, the double 
coset $P_\emptyset H$ is open in $G$ according to Proposition \ref{bigcell}, hence 
$\mathcal{C}_c^\infty(P_\emptyset\backslash P_\emptyset H,\mu_\chi)\subset Ind_{P_\emptyset}^G(\mu_\chi)$. Now here 
is the key observation of our argument: notice that if 
$P$ is the parabolic subgroup of type 

$$(n_1=1,\dots,n_{m-1}=1,n_m=2,n_{m+1}=1,\dots,n_{n-1}=1),$$ one has 
$PH=P_\emptyset H$, and as $P_\emptyset\backslash P$ is compact, one has 
$$\mathcal{C}_c^\infty(P\backslash PH,\mu_\chi)\subset \mathcal{C}_c^\infty(P_\emptyset\backslash P_\emptyset H,\mu_\chi).$$
But for $\phi\in \mathcal{C}_c^\infty(P\backslash PH,\mu_\chi)$, as $P_\emptyset\cap H\backslash P\cap H \simeq \D^*\backslash D^*$ 
is compact, one has 
$$\int_{P_\emptyset\cap H\backslash H} \phi(h)dh=vol(P_\emptyset\cap H\backslash P\cap H)\int_{P\cap H\backslash H} \phi(h)dh,$$ and in 
particular $L_{|\mathcal{C}_c^\infty(P\backslash PH,\mu_\chi)}$ is nonzero. As $PH$ is open in $G$ by Proposition \ref{bigcell} again, 
the space $\mathcal{C}_c^\infty(P\backslash PH,\mu_\chi)$ is contained in $ind_P^G(\mu_{\chi})$, hence $L$ does 
not vanish on $ind_P^G(\mu_{\chi})$, and this is absurd because $L$ comes from a linear form on $St(\chi)$. 
The multiplicity one statement is already a part of Proposition \ref{2poss}.
\end{proof}

\section{The case $d$ odd}\label{odd}

In this case $D\otimes_F E$ is a division algebra $D_E$ of index $d$ over its center $E$, and the 
Galois involution $\theta: z\mapsto \overline{z}$ extends to an involution 
$\theta:=Id\otimes \theta$ of $D_E$, which we will also write $x\mapsto \overline{x}$. In this case, 
$G=GL(m,D_E)$ and $H=GL(m,D)$. We denote by $O_{D_E}$ (resp. $O_D$) the ring of integers of 
$D_E$ (resp. $D$). We choose $\l\in E-F$, such that $\l^2 \in F$. We set $\nu_E(g)=|N_{rd,E}(g)|_E$ 
for $g\in GL(k,D_E)$. We also set $\nu_F(h)=|N_{rd,F}(h)|_F$ 
for $h\in GL(k,D)$. For $\chi$ a character of $E^*$, we denote again by $\mu_\chi$ the character 
$\chi \circ N_{rd,G}$, and its restriction to any subgroup of $G$.

\subsection{Representatives of $P\backslash G/H$}\label{rep2}

Let $P$ be the standard parabolic subgroup of $G$ corresponding to a partition $\overline{m}=(m_1,\dots,m_t)$ of $m$. Then 
$G/P$ identifies with the flags $\{0\}\subset V_1 \subset \dots \subset V_t=V$ of $V=D_E^m$, with 
$dim(V_i)_{D_E}=m_1+\dots+m_t$. We denote by $\mathcal{B}=(e_1,\dots,e_m)$ the canonical basis of $V$. The involution $\theta$ acts directly on $V$, and $H$ is the fixed points in $GL(V)_{D_E}$ of the 
involution $f\mapsto \theta\circ f\circ \theta$. In particular the situation differs only notationally 
from \cite{M11}, and all the results of Sections 3 and 4 there are still true in the more general situation studied here. 
For example if $\mathcal{F}=\{0\}\subset V_1 \subset \dots \subset V_t$ is a flag as above, for $i\leq j$, we denote by 
$S_{i,j}$ a complement of $V_i\cap \theta(V_{j-1})+\theta(V_{i-1})\cap V_j$ in $V_i\cap \theta(V_{j-1})$, which we choose 
$\theta$-stable if $i=j$. Setting $S_{j,i}=\theta(S_{i,j})$ for $i<j$, then 
$V=\oplus_{k,l} S_{k,l}$, and each $V_i$ decomposes as $\oplus_{k=1}^i\oplus_{l=1}^t S_{k,l}$. Two flags $\mathcal{F}$ 
and $\mathcal{F}'$ are in the same $H$-orbit if and only if $m_{i,j}=m'_{i,j}$, where $m_{i,j}=dim(S_{i,j})_{D_E}$ 
and $m'_{i,j}=dim(S'_{i,j})_{D_E}$ for all $i$ and $j$. We denote by 
$J(\overline{m})$ the set of symmetric matrices $(m_{i,j})_{i,j}$ of size $t\times t$ with positive integral entries, such that the 
sum of the $i$-th row is equal to $m_i$. To a matrix $s=(m_{i,j})_{i,j}$ in $J(\overline{m})$, we naturally have the sub-partition 
$(m_{1,1},m_{1,2},\dots,m_{r,r-1},m_{r,r})$ of $\overline{m}$ associated where we only take the nonzero 
$m_{i,j}$'s). For $s\in J(\overline{m})$, we denote by 
$\mathcal{B}_{i,j}$ the family $(e_{m_1+\dots+m_{i-1}+m_{i,1}+\dots+m_{i,j-1}+1},\dots,e_{m_1+\dots+m_{i-1}+m_{i,1}+\dots+m_{i,j}})$, 
and by $\mathcal{B}_{\{i,j\}}$ the family $\mathcal{B}_{i,j}\cup \mathcal{B}_{j,i}$ for $i<j$. In particular 
$\mathcal{B}=\mathcal{B}_{1,1}\cup \mathcal{B}_{1,2} \cup \dots \cup \mathcal{B}_{t,t-1} \cup \mathcal{B}_{t,t}$ is a basis of $V$. 
One then has the following result.

\begin{prop} \label{representatives2}
For $s\in J(\overline{m})$, let $u_s$ be the matrix which represents in $\mathcal{B}$, the linear map $v_s$ from $V$ to itself, 
which stabilises $Vect(\mathcal{B}_{i,i})_{D_E}$ and $Vect(\mathcal{B}_{\{i,j\}})_{D_E}$ for all $i\neq j$, and such that 
$Mat_{\mathcal{B}_{i,i}}((v_s)_{|Vect(\mathcal{B}_{i,i})_{D_E}})=I_{m_{i,i}}$, and 
$$Mat_{\mathcal{B}_{\{i,j\}}}((v_s)_{|Vect(\mathcal{B}_{\{i,j\}})_{D_E}})=\begin{pmatrix} I_{m_{i,j}} & -\l I_{m_{i,j}} \\
 I_{m_{i,j}} & \l I_{m_{i,j}} \end{pmatrix}.$$
The set $\{u_s,s\in J(\overline{m})\}$, form a set of representatives for $P\backslash G/H$. In particular 
the identity map of  $\{u_s,s\in J(\overline{m})\}$, induces a bijection from $P\backslash G/H$ to 
$P_{\overline{m}}(E)\backslash GL(m,E)/GL(m,F)$.
\end{prop} 

\begin{rem}\label{indep2}
In particular, as announced in Remark \ref{indep1}, for fixed $n=2m$, the set $I(\overline{n})$ is naturally a subset of $J(\overline{n})$, and thus one has an injection $w_s\mapsto u_s$ from 
$P_{\overline{n}}(E)\backslash GL(n,E)/GL(m,\mathbb{H})$ into $P_{\overline{n}}(E)\backslash GL(n,E)/GL(n,F)$. It would be nice to have a conceptual explanation for this. Notice that with our choices, 
the map $w_s\mapsto u_s$ sends the big cell to the small one.
\end{rem}

The matrix $w_s=u_s u_s^{-\theta}$ is a permutation matrix of order $2$. Writing 
$\llbracket 1,\dots,n \rrbracket$ as an ordered disjoint union 
$I_{1,1}\cup I_{1,2}\cup \dots \cup I_{r,r-1}\cup I_{r,r}$, with $I_{i,j}$ of length $m_{i,j}$, 
then the permutation associated to $w_s$ swaps $I_{i,j}$ and $I_{j,i}$ if $i<j$, preserving the order in those 
intervals, and acts as the identity on $I_{i,i}$. The group 
$u_s H u_s^{-1}$ is the fixed points of the involution $\sigma_s:g\mapsto w_s^{-1}\overline{g} w_s$. Again, 
the standard parabolic subgroup $P_s$ of $G$ associated to $s$ viewed as a sub-partition of 
$\overline{n}$, affords a useful decomposition of $P$.

\begin{prop}
For any $s\in J(\overline{m})$, one has $P\cap u_s H u_s^{-1}=P_s\cap u_s H u_s^{-1}$, and 
$P_s\cap u_s H u_s^{-1}$ is the semi direct product of $M_s\cap u_s H u_s^{-1}$ and $N_s\cap u_s H u_s^{-1}$.
\end{prop}

The elements of group $M_s^{\sigma_s}=M_s\cap u_s H u_s^{-1}$ are the matrices 
$diag(g_{1,1},g_{1,2},\dots,g_{r,r-1},g_{r,r})$, with $g_{j,i}\in G_{n_{j,i}}$ equal to $\theta(g_{i,j})$. We also have the same relation between modulus characters.

\begin{prop}\label{modulus2}
$(\d_{P_s^{\sigma_s}})_{|M_s^{\sigma_s}}=(\d_{P_s}^{1/2})_{|M_s^{\sigma_s}}.$
\end{prop}

\subsection{Non vanishing of invariant linear forms}\label{intertwiningperiods}

In this section we will show that the $H$-invariant linear form on $ind_{\B}^G(\1)$ does not vanish on $ind_{P}^G(\1)$ 
for some well chosen parabolic subgroup of $G$ containing $\B$ properly.\\ 
Let $s_0$ be the partition the element of $J(\overline{m})$ such that for all $i$, 
one has $m_{i,n+1-i}=1$. We denote by $u_0$ the matrix $u_{s_0}$, and by $w_0$ the matrix 
$w_{s_0}$ (it is the longest Weyl element). Then the double class $P_\emptyset u_0H$ is open in $G$ because 
$u_0^{-1}P_\emptyset u_0$ is $\theta$-split. We let $\d_s$ be the character $\d_{P_\emptyset}^{s}$. 
For $m=2$, the matrix $u_0$ is the matrix $\begin{pmatrix} 1 & -\l \\ 1 & \l \end{pmatrix}\in GL(2,D_E)$.\\ 


\noindent For $f\in ind_{P_\emptyset}^G(\1)$, we denote by 
$f_s$ the only element in $ind_{P_\emptyset}^G(\d_s)$ such that $f_s$ restricted to $K=G(O_F)=GL(m,O_{D_E})$ is 
equal to $f_{|K}$. If $\phi$ is the constant function equal to $1$ in $ind_{P_\emptyset}^G(\1)$, then $f_s$ 
is nothing else than $f\phi_s$. When $m=2$, we will write $\phi_2$ instead of $\phi$.

\begin{prop}\label{linearform}
For $f$ in $ind_{P_\emptyset}^G(\d_s)$, the integral $I_m(f_s)=\int_{u_0^{-1} {P_\emptyset} u_0\cap H\backslash H} f_s(u_0h)dh$ converges for $Re(s)$ large enough. 
Moreover, there is $Q\in \C[X]$ such that $Q(q^{-s})I_m(f_s)$ belongs to $\C[q^{\pm s}]$ for all $f\in ind_{P_\emptyset}^G(\1)$.
\end{prop}
\begin{proof}
This is a consequence of Theorems 2.8 and 2.16 of \cite{BD}, and the fact that the condition on "$\eta$" in 
[ibid.] is always satisfied by Theorem 4(i) of \cite{L}. In fact, in our particular situation, the general result of 
\cite{L} is not needed according to Remark 2.17 of \cite{BD}.
\end{proof}

\textit{From now on, and until the end of this paragraph, $m$ is even}. We start by 
the case $m=2$. In \cite{JLR}, this computation is done when $E$ is unramified over $F$. However, in Lemma 27 of \cite{JLR}, they explain another method which is in fact that of 7.6 in \cite{JL}, and which consists in writing the spherical vector as the integral of a Schwartz function. We recall it now.

\begin{prop}\label{m=2}
Suppose that $m=2$, then up to 
a unit in $\C[q^{\pm s}]$:
$$I_2(\phi_{2,s})=L(\1_{F^*},d(2s-1))/L(\eta_{E/F},2ds),$$ 
where $L$ is the usual Tate $L$-factor. In particular, $I_2(\phi_2)=I_2(\phi_{2,0})\neq 0$.
\end{prop}
\begin{proof}
Call $\Phi$ the characteristic map of the lattice 
$O_{D_E}^2$ in $D_E^2$. Then the integral $$\nu_E(g)^{ds}\int_{D_E^*}\Phi((0,t)g)\nu_E(t)^{2ds}dt$$
 converges absolutely 
for $Re(s)$ large enough by the theory of Godement-Jacquet Zeta integrals (\cite{GJ}), it is in fact an element of $\C(q^{-s})$, 
and one has $$\phi_{2,s}(g)=\nu_E(g)^{ds}\int_{D_E^*}\Phi((0,x)g)\nu_E(x)^{2ds}dx/L(\1_{E^*},2ds).$$ 
Call $\sigma$ the involution $g\mapsto w_0^{-1}\overline{g}w_0$ which fixes $u_0Hu_0^{-1}$. The integral above can be viewed as an integral over ${M_{\emptyset}}^\sigma$:
 $$\phi_{2,s}(g)=\nu_E(g)^{ds}\int_{{M_{\emptyset}}^\sigma}\Phi((0,1)tg)\nu_F(t)^{2ds}dt/L(\1_{E^*},2ds),$$ 
hence $$\phi_{2,s}(u_0h)=\int_{u_0^{-1}\B u_0\cap H}\Phi((0,1)u_0t'h)\nu_F(u_0t'h)^{2ds}dt'/L(\1_{E^*},2ds).$$
Integrating over $u_0^{-1}\B u_0\cap H\backslash H$, one finally gets 
$$I_2(\phi_{2,s})=\int_{H}\Phi((0,1)u_0h)\nu_F(u_0h)^{2ds}dh/L(\1_{E^*},2ds).$$ 
If $h=\begin{pmatrix} a & b \\ c & d \end{pmatrix}$, then $(0,1)u_0h$ is equal to 
$(a-\l c,b-\l d)$, but the conditions $a-\l c \in O_{D_E}$ and $b-\l d \in O_{D_E}$ mean that 
$(a,c)$ and $(b,d)$ belong to the same right $O_D$ lattice of $D^2$. Hence if we denote by $\Phi_0$ 
the characteristic function of $\mathcal{M}(2,O_D)$, there is 
$h_\l \in H$ such that $\Phi((0,1)u_0h)=\Phi(h_\l h).$ In particular, after a change of variable,
 there is $\e(s)\in \C[q^{\pm s}]^\times$ such that 
$$I_2(\phi_{2,s})=\e(s)\int_{H}\Phi_0(h)\nu_F(h)^{2ds}dh/L(\1_{E^*},2ds)=$$ 
$$=\e(s)L(\1_{H},2ds+(1-2d)/2)/L(\1_{E^*},2ds)$$
$$=\e(s)L(\1_{D^*},2ds+(1-3d)/2)L(\1_{D^*},2ds+(1-d)/2)/L(\1_{E^*},2s)$$
$$= \e(s)L(\1_{F^*},2ds-d)L(\1_{F^*},2ds)/L(\1_{E^*},2ds)=\e(s)L(\1_{F^*},2ds-d)/L(\eta_{E/F},2ds).$$
Here we used the inductivity relation of the Godement-Jacquet $L$-factor $L(\1_{H},s)$. This quantity does 
not vanish at $0$ because $\eta_{E/F}\neq \1_{F^*}$.
\end{proof}

The general case can be deduced from this one

\begin{prop}\label{nonzero}
For $m=2r$, let $P$ be the standard parabolic subgroup $G$ of type $\overline{m}=(1,\dots,1,2,1,\dots,1)$, with 
$m_{r/2}=2$. Then there is $f$ in $ind_{P}^G(\1)$ such $I_m(f\phi_s)=I_2(\phi_{2,s})$. In 
particular, taking $s=0$, one has $I_m(f)=I_2(\phi_2)\neq 0$.
\end{prop}
\begin{proof}
We set $u=u_0$. Let $\w$ be a uniformiser of $D_E$, and take $k$ large enough such that 
$L_k=H\cap u^{-1}(1+\w^k \mathcal{M}(m,O_{D_E}))u$ is a compact open subgroup of $GL(m,O_{D_E})$. We take $f$ in $\mathcal{C}_c(P\backslash PuH)$ which is the characteristic function $PuL_k$. We notice that both groups 
$u^{-1}\B u\cap H$ and $u^{-1}Pu\cap H$ are reductive, hence unimodular. Then, for $Re(s)$ large enough, one has 
$$I_m(f_s)=\int_{u^{-1}\B u\cap H\backslash H}f(uh)\phi_s(uh)dh$$
$$=\int_{u^{-1}Pu\cap H\backslash H}\int_{u^{-1}\B u\cap H\backslash u^{-1}Pu\cap H} f(uph)\phi_s(uph)dpdh$$
$$=\int_{u^{-1}Pu\cap H\backslash H}f(uh)(\int_{u^{-1}\B u\cap H\backslash u^{-1}Pu\cap H} \phi_s(uph)dp)dh$$
$$=\int_{u^{-1}Pu\cap L_k\backslash L_k}(\int_{u^{-1}\B u\cap H\backslash u^{-1}Pu\cap H} \phi_s(up)dp)dh$$
because $f$ is left $P$-invariant and $L_k\subset GL(m,O_{D_E})$. But the latter integral is a positive multiple of 
$$\int_{u^{-1}\B u\cap H\backslash u^{-1}Pu\cap H} \phi_s(up)dp,$$ 
which is in fact the integral $I_2(\phi_{2,s})$ on $GL(2,D)$ considered in Proposition \ref{m=2}. It thus follows 
from Proposition \ref{m=2} that $I_m(f)$ is nonzero.
\end{proof}

\subsection{Distinction of Steinberg representations}

We now proceed as in Section \ref{dist1}. First, we have the following proposition, which is proved 
in a similar manner to Proposition \ref{2poss}, so we only sketch the proof.

\begin{prop}\label{onedirection}
If $St(\chi)$ is distinguished, then $\chi_{|F^*}=\1$ or $\eta_{E/F}$, and 
$Hom_H(St(\chi),\1)$ is of dimension at most one. If $m$ 
is odd, then $\chi_{|F^*}=\1$.
\end{prop}
\begin{proof}
Suppose that $St(\chi)$ is distinguished, then a nonzero $H$-invariant 
linear form $L$ on $St(\chi)$ inflates to a linear form (still denoted $L$) on $ind_{\B}^G(\chi)$. Whether $m$ is even or odd, applying Mackey theory and Frobenius Reciprocity law, one obtains that $L$ does not vanish on 
$\mathcal{C}_c(\B\backslash \B u_0 H,\mu_\chi)\simeq ind_{u_0^{-1}\B u_0 \cap H}^H(\mu_\chi)$, as all other 
$H$-sub-quotients of $ind_{\B}^G(\chi)$ can't be distinguished. Frobenius Reciprocity for $ind_{u_0^{-1}\B u_0 \cap H}^H(\mu_\chi)$ allows to conclude on the value of $\chi_{|F^*}$ and on multiplicity one at the same time. 
\end{proof}

For any $s_0\in \C$, there is $l_{s_0}\in \Z$ such that the 
linear form \begin{equation}\label{l_0} L_{s_0}:f_{s_0}\mapsto \underset{s\rightarrow s_0}{lim}(1-q^{s_0-s})^{l_{s_0}} I_m(f_s)\end{equation} is nonzero on 
$ind_{P_\emptyset}^G(\d_{s_0})$, which is thus distinguished. In fact, for any character $\chi$ which restricts trivially 
to $F^*$, the linear map $L_{s_0}$ is still $H$-invariant on 
$\mu_\chi\otimes ind_{P_\emptyset}^G(\d_{s_0})=ind_{P_\emptyset}^G(\mu_\chi\d_{s_0})$. We deduce as in the 
proof of Proposition \ref{infact1}, the following statement, the proof of which we sketch again. 

\begin{thm}\label{odd}
If $m$ is odd and $\chi_{|F^*}=\1$, then $St(\chi)$ is distinguished, hence $St(\chi)$ is distinguished if and only if $\chi_{|F^*}=\1$.
\end{thm}
\begin{proof}
Suppose that $m$ is odd, and $\chi_{|F^*}=\1$. As in the proof of Proposition \ref{infact1}, one sees, using Frobenius reciprocity, that 
every representation $ind_{P}^G(\mu_\chi)$ can't be distinguished, for any standard parabolic subgroup $P$ of type 
$\overline{n}$, with all $n_i$'s equal to $1$, except one which is $2$. The linear form $L_0$ ($L_{s_0}$ with $s_0=0$) thus descends to $St(\chi)$, which is thus distinguished.
\end{proof}

Now we focus on the even case. The following is proved again just as Proposition \ref{infact1}, we omit the proof.

\begin{prop}\label{versace}
If $m$ is even, and $\chi_{|F^*}=\eta_{E/F}$, then $St(\chi)$ is distinguished.
\end{prop}

Finally, we obtain, thanks to the results of Section \ref{intertwiningperiods}, the main result 
when $m$ is even.

\begin{thm}\label{main2}
If $m$ is even, then $St(\chi)$ is distinguished if and only if $\chi_{|F^*}=\eta_{E/F}$.
\end{thm}
\begin{proof}
It remains to show that if $St(\chi)$ is distinguished, then $\chi_{|F^*}=\eta_{E/F}$. According to Proposition 
\ref{onedirection}, it is enough to show that if  $\chi_{|F^*}=\1$, then $St(\chi)$ isn't distinguished. In order to obtain 
a contradiction, suppose that it is. Then the linear form on $St(\chi)$ inflates to $ind_{\B}^G(\mu_\chi)$, hence must be equal (up to a nonzero scalar) 
to $L_0$, because $Hom_H(ind_{\B}^G(\mu_\chi),\1)$ is one dimensional according to the proof of Proposition \ref{onedirection}. 
Moreover, we also know from the proof of Proposition \ref{onedirection} that $L_0$ restricts non trivially to 
$\mathcal{C}_c(\B\backslash \B u_0 H,\mu_\chi)$. In particular, the integer $l_0$ (see before Theorem \ref{odd}) must be equal to $0$, i.e. 
one has $L_0=I_m$ up to a nonzero scalar. Now if we apply Proposition \ref{nonzero}, we obtain that $L_0(\mu_\chi \otimes f)=
I_m(\mu_\chi \otimes f)=\mu_\chi(u_0)I_m(f)\neq 0$. This is absurd 
as $L_0$ must vanish on $ind_P^G(\mu_\chi)$ (for $P$ as in Proposition \ref{nonzero}), because it descends to $St(\chi)$.
\end{proof}

\begin{rem}
When $F$ has charactersitic zero, it is a consequence of the global results in \cite{F} and \cite{FH} that the inverse of the Jacquet-Langlands correspondence (\cite{DKV}, \cite{Bad}) sends distinguished cuspidal representations of $GL(md,E)$ to distinguished representations of $GL(m,D\otimes E)$. Without restriction on the characteristic, when the cuspidal representation has level zero, there is also an explicit proof of this result using type theory in \cite{Con}. It follows at once from Theorem 3.15 of \cite{Bad07} applied to the trivial representation, that the result of this paper says that a Steinberg representation of $GL(m,D\otimes E)$ is distinguished if and only if its image by the Jacquet-Langlands correspondence is.
\end{rem}

\end{document}